\documentclass[12pt,leqno]{amsart}
\usepackage{latexsym,amsmath,amssymb,amscd,xcolor,mathrsfs,marvosym,mathtools}
\usepackage[T1]{fontenc}
\usepackage[utf8]{inputenc}
\usepackage{graphicx,hyperref}
\usepackage[shortlabels]{enumitem}

plus 6pt minus 12pt
\newtheorem{theorem}{Theorem}
\newtheorem*{theorem*}{Theorem}
\newtheorem{lemma}[theorem]{Lemma}
\newtheorem{corollary}[theorem]{Corollary}
\newtheorem{proposition}[theorem]{Proposition}

\newtheorem{conjecture}[theorem]{Conjecture}

\theoremstyle{definition}
\newtheorem{remark}[theorem]{Remark}

\newtheorem{example}[theorem]{Example}

\newcommand{\overbarint}{
\rule[.036in]{.12in}{.009in}\kern-.16in \displaystyle\int }

\newcommand{\overbarcal}{\mbox{$ \rule[.036in]{.11in}{.007in}\kern-.128in\int $}}

\newcommand{\bbbn}{\mathbb N}

\newcommand{\bbbr}{\mathbb R}
\newcommand{\eps}{\varepsilon}
\newcommand{\bbbs}{\mathbb S}
\newcommand{\bbbb}{\mathbb B}

\newcommand{\la}{\langle}
\newcommand{\ra}{\rangle}
\newcommand{\we}{\vec{e}}

\def\5{\text{\Saturn}}

\def\diam{\operatorname{diam}}

\def\dist{\operatorname{dist}}

\def\rank{{\rm rank\,}}
\def\id{{\rm id\, }}

\newcommand{\overbar}[1]{\mkern 1.7mu\overline{\mkern-1.7mu#1\mkern-1.5mu}\mkern 1.5mu}

\def\mvint_#1{\mathchoice
          {\mathop{\vrule width 6pt height 3 pt depth -2.5pt
                  \kern -8pt \intop}\nolimits_{\kern -3pt #1}}%
          {\mathop{\vrule width 5pt height 3 pt depth -2.6pt
                  \kern -6pt \intop}\nolimits_{#1}}%
          {\mathop{\vrule width 5pt height 3 pt depth -2.6pt
                  \kern -6pt \intop}\nolimits_{#1}}%
          {\mathop{\vrule width 5pt height 3 pt depth -2.6pt
                  \kern -6pt \intop}\nolimits_{#1}}}

\numberwithin{theorem}{section} \numberwithin{equation}{section}

\title[Approximation of mappings with $\rank Df\leq 1$]{Smooth approximation of mappings with rank of the derivative at most $1$}
\author[P. Goldstein]{Pawe\l{}  Goldstein}
\address{Pawe\l{} Goldstein, Institute of Mathematics, Faculty of Mathematics, Informatics and Mechanics, University of Warsaw, Banacha 2, 02-097 Warsaw, Poland} \email{P.Goldstein@mimuw.edu.pl}
\thanks{P.G. was supported by NCN grant no 2019/35/B/ST1/02030}
\author[P. Haj\l{}asz]{Piotr Haj\l{}asz}
\address{Piotr Hajlasz, Department of Mathematics, University of Pittsburgh, Pittsburgh, PA 15260, USA}
\email{hajlasz@pitt.edu}
\thanks{P.H.\ was supported by NSF grant DMS-2055171}

\begin{document}


\maketitle
\begin{abstract}
It was conjectured that if $f\in C^1(\bbbr^n,\bbbr^n)$ satisfies $\rank Df\leq m<n$ everywhere in $\bbbr^n$, then $f$ can be uniformly approximated by $C^\infty$-mappings $g$ satisfying $\rank Dg\leq m$ everywhere. While in general, there are counterexamples to this conjecture, we prove that the answer is in the positive when $m=1$. More precisely, if $m=1$, our result yields an almost-uniform approximation of locally Lipschitz mappings $f:\Omega\to\bbbr^n$, satisfying $\rank Df\leq 1$ a.e., by $C^\infty$-mappings $g$ with $\rank Dg\leq 1$, provided $\Omega\subset\bbbr^n$ is simply connected. The construction of the approximation employs techniques of analysis on metric spaces, including the theory of metric trees ($\bbbr$-trees).
\end{abstract}
\section{Introduction}

The following conjecture was stated by Jacek Ga\l{}ęski \cite[Conjecture~1.1 and Section~3.3]{galeski}.

\begin{conjecture}
\label{C1}
Let $1\leq m<n$ be integers and let $\Omega\subset\bbbr^n$ be open. If $f\in C^1(\Omega,\bbbr^n)$ satisfies $\rank Df\leq m$ everywhere in $\Omega$, then $f$ can be uniformly approximated by smooth mappings $g\in C^\infty(\Omega,\bbbr^n)$ such that $\rank Dg\leq m$ everywhere in $\Omega$.
\end{conjecture}
One can also formulate a weaker, local version of this conjecture (see \cite{GH1}).
\begin{conjecture}
\label{C2}
Let $1\leq m<n$ be integers and let $\Omega\subset\bbbr^n$ be open. If $f\in C^1(\Omega,\bbbr^n)$ satisfies $\rank Df\leq m$ everywhere in $\Omega$, then for every $x\in\Omega$ there is a neighborhood $\bbbb^n(x,\eps)\subset\Omega$ and a sequence $f_i\in C^\infty(\bbbb^n(x,\eps),\bbbr^n)$ such that $\rank Df_i\leq m$ everywhere in $\bbbb^n(x,\eps)$ and $f_i$ converges uniformly to $f$ on $\bbbb^n(x,\eps)$.
\end{conjecture}
These are very natural conjectures and the main difficulty is that standard approximation techniques like the one based on convolution do not preserve the rank of the derivative. It is a highly nonlinear constraint, difficult to deal with.

However, there is an open and  dense subset $G\subset\Omega$, where the rank of the derivative is locally constant. Using the rank theorem on $G$ along with the standard approximation by convolution one easily obtains a partial result in the positive direction \cite[Theorem~3]{GH1}.
\begin{theorem}
Let $1\leq m<n$ be integers and let $\Omega\subset\bbbr^n$ be open. If $f\in C^1(\Omega,\bbbr^n)$ satisfies $\rank Df\leq m$ everywhere in $\Omega$, then there is an open and dense set $G\subset\Omega$ such that for every $x\in G$ there is a neighborhood $\bbbb^n(x,\eps)\subset G$ and a sequence ${f_i\in C^\infty(\bbbb^n(x,\eps),\bbbr^n)}$, such that $\rank Df_i\leq m$ in $\bbbb^n(x,\eps)$ and $f_i$ converges to $f$ uniformly on $\bbbb^n(x,\eps)$.
\end{theorem}
The problem is, however, caused by the closed and nowhere dense set $\Omega\setminus G$, where the rank of the derivative is not constant and the rank theorem cannot be used. In fact, in \cite{GH1} the authors constructed infinitely many counterexamples to Conjecture~\ref{C2} and hence also to Conjecture~\ref{C1}.
\begin{example}
There is $f\in C^1(\bbbr^5,\bbbr^5)$ with $\rank Df\leq 3$ that cannot be locally and uniformly approximated (in the sense of Conjecture~\ref{C2}) by mappings
$g\in C^2(\bbbr^5,\bbbr^5)$ satisfying $\rank Dg\leq 3$.
\end{example}
\begin{example}
There is $f\in C^1(\bbbr^7,\bbbr^7)$, $\rank Df\leq 4$, that cannot be locally and uniformly approximated (in the sense of Conjecture~\ref{C2}) by mappings
$g\in C^3(\bbbr^7,\bbbr^7)$ satisfying $\rank Dg\leq 4$.
\end{example}
These examples are special cases of a much more general result \cite[Theorem~4]{GH1}, which provides infinitely many similar examples.
\begin{theorem}
\label{T25}
Suppose that $m+1\leq k<2m-1$, $\ell\geq k+1$, $r\geq m+1$, and the homotopy group $\pi_k(\bbbs^m)$ is non-trivial. Then there is a map $f\in C^1(\bbbr^\ell, \bbbr^r)$ with $\rank Df\leq m $ in $\bbbr^\ell$ and a Cantor set $E\subset \bbbr^\ell$ with the following property:
\begin{quote}
For every $x_o\in E$ and $\eps>0$ there is $\delta>0$ such that\\
if ${g\in C^{k-m+1}(\bbbb^\ell(x_o,\eps),\bbbr^r)}$ and
$|f(x)-g(x)|<\delta \text{ for all } x\in\bbbb^\ell(x_o,\eps)$,\\
then $\rank Dg\geq m+1$ on a non-empty open set in $\bbbb^\ell(x_o,\eps)$.
\end{quote}
\end{theorem}
The proof of the theorem involves the methods of algebraic topology. 

In the same paper \cite[Conjecture~6]{GH1} the authors conjectured that if $m=1$, then in fact the uniform approximation is possible. The main result of the paper answers the conjecture in the positive: 
\begin{theorem}
\label{T26}
Assume $\Omega\subset\bbbr^n$ is a simply connected domain. 
If $f:\Omega\to\bbbr^m$ is a locally Lipschitz map satisfying $\rank Df\leq 1$ a.e. in $\Omega$, then there exist $C^\infty$-smooth maps $f_i:\Omega\to\bbbr^m$ with $\rank Df_i\leq 1$ in $\Omega$, such that for any compact $K\subset \Omega$ the functions $f_i$ converge uniformly to $f$ on $K$, as $i\to\infty$.
\end{theorem}

An immediate consequence addresses Conjecture \ref{C2}:
\begin{corollary}
Let $\Omega\subset\bbbr^n$ be open. If $f:\Omega\to\bbbr^m$ is a locally Lipschitz map satisfying $\rank Df\leq 1$ a.e. in $\Omega$, then for every $x\in\Omega$ there is a neighborhood $\bbbb^n(x,\eps)\subset\Omega$ and $C^\infty$-smooth maps $f_i:\bbbb^n(x,\eps)\to \bbbr^n$, such that $\rank Df_i\leq 1$ everywhere in $\bbbb^n(x,\eps)$, and the functions $f_i$~converge uniformly to $f$ on $\bbbb^n(x,\eps)$, as $i\to\infty$.
\end{corollary}

While Theorem \ref{T26} has a purely Euclidean statement, it is interesting to note that the proof presented in the paper is based mostly on quite abstract techniques of analysis on metric spaces. However, since the result should be of interest to researchers that are not familiar with analysis on metric spaces, the proofs presented in the paper are self-contained.

A very rough idea of the proof is as follows. 

Assume first that $f\in C^1$.
The set $G\subset\Omega$ where $\rank Df=1$ is open. According to the rank theorem a neighborhood of every point $x\in G$ is mapped onto a $C^1$ curve. However, at points of the closed set where $\rank Df=0$, the $C^1$ curve may branch into infinitely many $C^1$ curves. Thus a rough intuition is that a mapping $f\in C^1$ satisfying the rank condition $\rank Df\leq 1$ is a $C^1$ curve with possibly infinitely many branching points. The situation is very complicated from the topological point of view, and the best way to approach it is to construct an abstract metric space $Z_f$ (a metric tree) that is a representation of the branching of the map.  Roughly speaking, edges of $Z_f$ correspond to $C^1$ curves in the image of $G$.
It is possible to construct $Z_f$ in such a way that the map $f$ factors through~$Z_f$, 
$$
\Omega\xrightarrow{\psi}Z_f\xrightarrow{\phi}\bbbr^m, \quad f=\phi\circ\psi,
$$
where $\psi$ is locally Lipschitz and $\phi$ is Lipschitz. 
In fact, the construction of the metric tree $Z_f$ along with the factorization $f=\phi\circ\psi$ is possible if $f$ is Lipschitz and such that $\rank Df\leq 1$ almost everywhere.

The construction of the space $Z_f$ is quite abstract and the proof that $Z_f$ is a metric tree requires intricate arguments from analysis on metric spaces.

Then, we construct a retraction $r:Z_f\to T$ onto a finite sub-tree, which is uniformly close to the identity map $\id:Z_f\to Z_f$, so that $\phi\circ r\circ\psi$ is an approximation of $f=\phi\circ\psi$.

We can assume that the finite tree $T$ is embedded into $\bbbr^E$, where $E$ is the number of edges in $T$.
Next, we approximate $r\circ\psi:\Omega\to T\subset\bbbr^E$ by $C^\infty$-mappings
$g_i:\Omega\to T\subset\bbbr^E$, $g_i\in C^\infty(\Omega,\bbbr^E)$, so $\phi\circ g_i$ is close to $\phi\circ r\circ\psi$ and hence close to $f=\phi\circ\psi$.

Clearly, the mappings $g_i$ satisfy $\rank Dg_i\leq 1$, because their images lie in the finite tree~$T$.

We can extend $\phi|_T:T\to\bbbr^m$ to a Lipschitz map $\tilde{\phi}:\bbbr^E\to\bbbr^m$ and we can approximate $\tilde{\phi}$ by $C^\infty$-mappings $\phi_i\in C^\infty(\bbbr^E,\bbbr^m)$.

Finally, $\phi_i\circ g_i:\Omega\to\bbbr^m$ is $C^\infty$-smooth, $\rank D(\phi_i\circ g_i)\leq 1$ by the chain rule, and $\phi_i\circ g_i$ is close to $\tilde{\phi}\circ g_i=\phi\circ g_i$, which is close to $f$.

The method of factorization through metric trees used in the proof of Theorem~\ref{T26} is very different and completely unrelated to the methods of algebraic topology used in the proof of Theorem~\ref{T25}. However, quite surprisingly, both techniques have originally been used in \cite{WengerY} as tools for study of Lipschitz homotopy groups of the Heisenberg group, a problem that seems completely unrelated to Theorems~\ref{T25} and~\ref{T26}. These techniques were elaborated in \cite{DavidS,EsmayliH,GH1,GHP,Zust}.
In particular, we use results from \cite{EsmayliH}, and some of the approximation techniques from \cite{GHP}.
See also \cite{LW,PS} for related constructions.

In fact, the idea of factorization of real valued functions through graphs and trees already appeared in the work of Kronrod \cite{kronrod} and Reeb \cite{reeb}. 
This construction known as the Kronrod-Reeb graph or Reeb graph, has many applications in computational geometry and in topological data analysis, see  e.g., \cite{EH,MO} and references therein.

The paper is structured as follows. Sections~\ref{S1}, \ref{S5}, and~\ref{S2} are devoted to auxiliary material needed in the proof of Theorem~\ref{T26}. In Section~\ref{S1} we discuss some topics on analysis in metric spaces, including the theory of metric trees. In Section~\ref{S5} we discuss some results about the rank of the derivative of Lipschitz maps, and in Section~\ref{S2} we discuss factorization of Lipschitz maps through metric trees. The final Section~\ref{S3} is entirely devoted to the proof of Theorem~\ref{T26}.

\subsection*{Notation}
The notation in the paper is rather standard, but we list here the conventions used for the convenience of the reader.

A \emph{domain} in $\bbbr^n$ is an open and connected subset of $\bbbr^n$. By $\bbbb^n(x,r)$ we denote the Euclidean open ball centered at $x$, of radius $r$. If no center or radius is specified, $\bbbb^n$ denotes the unit ball in $\bbbr^n$. Similarly, $\bbbs^n$ denotes the unit $n$-dimensional sphere.

The Euclidean norm in $\bbbr^k$ is denoted by $|\cdot |$ (for any $k$) and for $x,y\in\bbbr^k$ we denote by $x\cdot y$ the Euclidean scalar product of $x$ and $y$. If $A$ is a linear map from $\bbbr^k$ to $\bbbr^m$, its \emph{operator norm} is defined as $\|A\|=\sup\{|Ax|~~:~~|x|\leq 1\}$.

If $X$ is a subset of a metric space, then by $\overbar{X}$ we denote the closure of $X$, and for any $\delta>0$  the $\delta$-neighborhood of $X$ is the set of points at distance less than $\delta$ from $X$.

For a given $\eps>0$, a subset $X$ of a metric space $Y$ is an \emph{$\eps$-net}, if $\dist(y,X)<\eps$ for every $y\in Y$. If $X$ is compact, then for any $\eps>0$ there exists a \emph{finite} $\eps$-net in $X$.

\section{Metric spaces}
\label{S1}
In this section we briefly recall some facts from the theory of metric spaces. 
While the material is mostly well known, some of the results (Proposition~\ref{T23} and Theorem~\ref{T28}) seem new. 

\subsection{Rectifiable curves}
For more details see e.g., \cite[Section 3]{H2003}.

Let $(X,d)$ be a metric space. By a curve in $X$ we mean a continuous map $\gamma:[a,b]\to X$. The \emph{length} of $\gamma$ is defined as
$
\ell(\gamma)=\sup\sum_{i=0}^{n-1} d(\gamma(t_i),\gamma(t_{i+1})),
$
where the supremum is taken over all partitions $a=t_0<t_1<\ldots<t_n=b$.

A curve is \emph{rectifiable} if $\ell(\gamma)<\infty$. Clearly, if $\gamma(a)=x$ and $\gamma(b)=y$, then $d(x,y)\leq \ell(\gamma)$.

Reparameterization of a curve does not change its length and every rectifiable curve can be reparameterized in such a way that the reparameterized curve $\hat{\gamma}:[0,\ell(\gamma)]\to X$ satisfies $\ell(\hat{\gamma}|_{[s,t]})=t-s$ for all $0\leq t<s\leq b$. In particular, $\hat{\gamma}$ is $1$-Lipschitz. Such an orientation preserving reparameterization is unique and we say that $\hat{\gamma}$ is parameterized by \emph{arc-length}, see e.g., \cite[Theorem 3.2]{H2003}.

For us, the only important consequence of this fact is that when studying rectifiable curves in $X$, we may restrict ourselves to the class of Lipschitz continuous curves. 

We say that a metric space is \emph{proper} if bounded and closed sets are compact. 
\begin{lemma}
\label{T1}
If a metric space $X$ is proper and there is a rectifiable curve connecting $x,y\in X$, then there is a shortest curve connecting $x$ and $y$.
\end{lemma}
For a proof, see e.g., \cite[Theorem 3.9]{H2003}. This is a simple consequence of the Arzel\`a-Ascoli theorem and the lower semicontinuity of the length with respect to the uniform convergence of curves.

A shortest curve connecting $x$ to $y$ is called a \emph{geodesic}. Geodesics are not necessarily unique. 
\begin{lemma}
\label{T14}
Any geodesic is one-to-one.
\end{lemma}
\begin{proof}
Otherwise we could make it shorter by removing `loops'. 
\end{proof}

We say that $X$ is a {\em length space} if for any $x,y\in X$, $d(x,y)$ equals the infimum of lengths of curves connecting $x$ to $y$. A space is {\em geodesic} if for any $x,y\in X$, there is a curve $\gamma$ connecting $x$ to $y$ such that $d(x,y)=\ell(\gamma)$. 

Clearly, every geodesic space is a length space. However, $\bbbr^n\setminus \{ 0\}$ is a length space, but not a geodesic one. It is also locally compact, but not proper. The next result is an immediate consequence of Lemma~\ref{T1}.
\begin{lemma}
\label{T2}
If $X$ is a proper length space, then $X$ is geodesic.
\end{lemma}

An {\em arc} is a metric space homeomorphic to the interval $[0,1]$. 

If $\Gamma$ is an arc and $\gamma_i:[a_i,b_i]\to \Gamma$, $i=1,2$, are homeomorphisms, then $\ell(\gamma_1)=\ell(\gamma_2)$ and we define $\ell(\Gamma)$ as the length of any of its homeomorphic parameterizations. We say that an arc $\Gamma$ is {\em rectifiable} if $\ell(\Gamma)<\infty$.

It follows from Lemma~\ref{T14} that the image of a geodesic is a rectifiable arc.

\begin{lemma}
\label{T16}
If points $x,y\in X$, $x\neq y$, can be connected by a rectifiable curve $\gamma$ and if $\Gamma$ is the image of the curve, then there is a rectifiable arc inside $\Gamma$ with endpoints $x$ and $y$, of length less than or equal $\ell(\gamma)$.
\end{lemma}
\begin{proof}
$\Gamma$ is compact and $x,y\in \Gamma$ can be connected by a rectifiable curve in $\Gamma$, so by Lemma~\ref{T1} there is a shortest curve inside $\Gamma$ connecting $x$ to $y$. By Lemma~\ref{T14} the curve is one-to-one and hence its image is an arc.
\end{proof}
\begin{corollary}
\label{T17}
In a length space $X$, for any distinct points $x,y\in X$, $d(x,y)$ equals the infimum of lengths of rectifiable arcs with endpoints $x$ and $y$.
\end{corollary}
\begin{lemma}
\label{T29}
Suppose that $\Gamma_1$ and $\Gamma_2$ are rectifiable arcs in a metric space $X$ connecting the same endpoints. If $\Gamma_1\neq \Gamma_2$, then there is a one-to-one Lipschitz map $\alpha:\bbbs^1\to X$.
\end{lemma}
\begin{proof}
Let $\gamma_{1,2}:[0,1]\to \Gamma_{1,2}$ be Lipschitz homeomorphisms with
$x:=\gamma_1(0)=\gamma_2(0)$, 
$y:=\gamma_1(1)=\gamma_2(1)$.
(Note that the homeomorphisms $\gamma_{1,2}$ are not necessarily bi-Lipschitz.)
If $\Gamma_1\neq\Gamma_2$, say $\Gamma_1\setminus\Gamma_2\neq\varnothing$, then there is $t_0\in (0,1)$ such that $\gamma_1(t_0)\not\in\Gamma_2$, and  we can find $a\in [0,t_0)$ and $b\in (t_0,1]$ such that
$$
\tilde{x}:=\gamma_1(a)\in\Gamma_2,
\quad
\tilde{y}:=\gamma_1(b)\in\Gamma_2,
\quad
\gamma_1(a,b)\cap\Gamma_2=\varnothing.
$$
Now, concatenation of subarcs of $\Gamma_1$ and $\Gamma_2$ between the points $\tilde{x}$ and $\tilde{y}$ defines a one-to-one Lipschitz map 
$\alpha:\bbbs^1\to X$.
\end{proof}
\begin{remark}
Although the map $\gamma:\bbbs^1\to X$ is one-to-one and Lipschitz, it need not be bi-Lipschitz, since its image may possess `cusps'.
\end{remark}

\subsection{Lipschitz functions}
Let $X$ be a metric space. A mapping $f:\Omega\to X$, $\Omega\subset\bbbr^n$ open, is \emph{locally Lipschitz}, if for every $x\in\Omega$ there is $r>0$ such that $f|_{B(x,r)}$ is Lipschitz.

\begin{lemma}
\label{T3}
Let $f:\Omega\to X$ be a mapping from an open set $\Omega\subset\bbbr^n$ to a metric space $X$. Then the following conditions are equivalent:
\begin{enumerate}[(1)]
\item[(a)] $f$ is locally Lipschitz,
\item[(b)] $f|_K$ is Lipschitz for every compact set $K\subset \Omega$,
\end{enumerate}
\end{lemma}
\begin{proof}
The implication from (b) to (a) is obvious, so it remains to show that (a) implies (b).

Clearly, $M:=\sup_{x,y\in K} d(f(x),f(y))<\infty$. Let $\{B_i\}_{i=1}^N$ be a finite covering of $K$ by balls such that for each $i$, $f|_{B_i}$ is Lipschitz. Let $L$ be the maximum of the Lipschitz constants of these functions. Let $\delta>0$ be the Lebesgue number of the covering. 

If $x,y\in K$ and $|x-y|<\delta$, then $x$ and $y$ belong to one of the balls $B_i$ and hence $d(f(x),f(y))\leq L|x-y|$. 
If $|x-y|\geq \delta$, then $d(f(x),f(y))\leq M\leq M\delta^{-1} |x-y|$. 
This proves that $f|_K$ is Lipschitz.
\end{proof}

The following extension result is due to McShane \cite{McShane}, \cite[Theorem~6.2]{heinonen}.
\begin{lemma}
\label{T19}
Let $X$ be a metric space and let $f:E\to\bbbr$ be an $L$-Lipschitz function defined on a subset $E$ of $X$. Then there exists an $L$-Lipschitz function $F:X\to\bbbr$ such that $F(x)=f(x)$ for all $x\in E$.
\end{lemma}

The next result is also well known, see \cite[Theorem~6.8]{heinonen}.
\begin{lemma}
\label{T20}
If $f:X\to\bbbr$ is a bounded and uniformly continuous function on a metric space, then there is a sequence of Lipschitz continuous functions $f_i:X\to\bbbr$, $i=1,2,3,\ldots$, such that $f_i\to f$ converges uniformly on $X$.
\end{lemma}

 \subsection{Mapping spheres into metric spaces}
\begin{proposition}
\label{T23}
Let $X$ be a metric space. Assume that there is a continuous map $\alpha:\bbbs^n\to X$, and $x_o\in\bbbs^n$, such that $\alpha$ is one-to-one in a neighborhood of $x_o$, and $\alpha^{-1}(\alpha(x_o))=\{x_o\}$. Then there is a Lipschitz map $\pi:X\to\bbbr^{n+1}$ such that $(\pi\circ\alpha)(\bbbs^n)=\bbbs^n\subset\bbbr^{n+1}$, and the map $\pi\circ\alpha:\bbbs^n\to\bbbs^n$ is homotopic to the identity map.
\end{proposition}
\begin{proof}
Without loss of generality, we may assume that $x_o=N$ is the north pole of $\bbbs^n$. It follows from the assumptions of the theorem that there is a closed spherical cap $W$ centered at $N$ (boundary of $W$ is a parallel of constant longitude), such that $\alpha:{W}\to\alpha({W})$ is a homeomorphism and
$\alpha(\bbbs^n\setminus{W})\cap\alpha({W})=\varnothing$.

Let $R:\bbbs^n\to\bbbs^n$ be the continuous map that stretches ${W}$ onto $\bbbs^n$ along meridians, and maps $\bbbs^n\setminus W$ to the south pole $S$. Clearly, $R$ is homotopic to the identity.

The map $\alpha^{-1}:\alpha({W})\to{W}$ is continuous and we define $g:\alpha(\bbbs^n)\to\bbbs^n$ by
$$
g(y)=\begin{cases}
(R\circ\alpha^{-1})(y) & \text{if } y\in\alpha({W}),\\
S                   & \text{if } y\in\alpha(\bbbs^n)\setminus\alpha({W}).
\end{cases}
$$
It is easy to see that the map $g$ is continuous and that $g\circ\alpha=R:\bbbs^n\to\bbbs^n$, so $g\circ\alpha$ is homotopic to the identity.

According to Lemma~\ref{T20}, we can find a Lipschitz map $h:\alpha(\bbbs^n)\to\bbbr^{n+1}$, such that
$$
|h(y)-g(y)|\leq 1/2
\quad
\text{for all } y\in\alpha(\bbbs^n).
$$
Since $|g(y)|=1$, it follows that $|h(y)|\geq 1/2$, so the map
$$
\pi(y)=\frac{h(y)}{|h(y)|}:\alpha(\bbbs^n)\to\bbbs^n
$$
is Lipschitz continuous and $|\pi(y)-g(y)|\leq 1$ for all $y\in\alpha(\bbbs^n)$. Therefore, the maps
$\pi,g:\alpha(\bbbs^n)\to\bbbs^n$ are homotopic (homotopy along unique shortest geodesics in $\bbbs^n$). This also implies that
$\pi\circ\alpha:\bbbs^n\to\bbbs^n$ is homotopic to $g\circ\alpha=R$, and hence $\pi\circ\alpha$ is homotopic to the identity map.

Now, it remains to extend $\pi:\alpha(\bbbs^n)\to\bbbs^n\subset\bbbr^{n+1}$ to a Lipschitz map $\pi:X\to\bbbr^{n+1}$ using Lemma~\ref{T19}.
\end{proof}

The next result illustrates the above proposition. We learned it from Petrunin, see \cite[p.67]{Petrunin}.
\begin{corollary}
\label{T24}
Let $X$ be a contractible metric space with zero $(n+1)$-dimensional Hausdorff measure. Assume that $D_1,D_2\subset X$ are two embedded closed $n$-dimensional balls having the same boundary. Then $D_1=D_2$.
\end{corollary}
\begin{proof}
Suppose by way of contradiction that $D_1\neq D_2$. 
Since $D_1$ and $D_2$ have common boundary, there is a continuous map $\alpha:\bbbs^n\to X$ that is a homeomorphism of the upper hemisphere $\bbbs^n_+$ and the lower hemisphere $\bbbs^n_-$ onto $D_1$ and $D_2$ respectively. 

Indeed, we can construct $\alpha$ as follows. Let $\alpha_+:\bbbs^n_+\to D_1$ and $\alpha_-:\bbbs^n_-\to D_2$ be homeomorphisms. We cannot glue them along the equator $\bbbs^{n-1}=\bbbs^n_+\cap\bbbs^n_-$, because $\alpha_+$ and $\alpha_-$ need not agree on $\bbbs^{n-1}$. Note that 
$\alpha_-^{-1}\circ\alpha_+:\bbbs^{n-1}\to\bbbs^{n-1}$ is a homeomorphism and we can extend it to a homeomorphism of the lower hemisphere $h:\bbbs^n_-\to\bbbs^n_-$ as a map that maps each of the $(n-1)$-dimensional spheres parallel to the equator $\bbbs^{n-1}$ onto itself as a scaled copy of the homeomorphism $\alpha_-^{-1}\circ\alpha_+$. Then
$$
\alpha=
\begin{cases}
\alpha_+          & \text{on } \bbbs^n_+\\
\alpha_-\circ h   & \text{on } \bbbs^n_-
\end{cases}
$$
has all the properties we need.

Since $D_1\neq D_2$, it follows that the mapping $\alpha$ satisfies the assumptions of Proposition~\ref{T23}. By contractibility of the space, the map $\alpha$ admits a continuous extension to $A:\bbbb^{n+1}\to X$. Let $\pi:X\to\bbbr^{n+1}$ be the mapping from Proposition~\ref{T23}. Then $\bbbb^{n+1}\subset (\pi\circ A)(\bbbb^{n+1})\subset \pi(X)$. That shows that a Lipschitz image of $X$ has positive $(n+1)$-dimensional Hausdorff measure and hence $X$ has positive $(n+1)$-dimensional Hausdorff measure, which is a contradiction.
\end{proof}

\subsection{Metric trees}
A geodesic space is called a {\em metric tree} if for any $x,y\in X$, $x\neq y$, there is a unique arc with endpoints $x$ and $y$.

The next lemma is well known, see e.g., \cite[Lemma~2.2.2]{chiswell}.
\begin{lemma}
\label{T30}
Metric trees are contractible.
\end{lemma}
Indeed, we fix a point $x_o$ in a metric tree and we perform a contraction to $x_o$ of any point $x$ along the unique geodesic connecting $x_o$ to point $x$. One only needs to check that the homotopy created this way is continuous.

\begin{lemma}
\label{T27}
If $X$ is a length space, and for any $x,y\in X$, $x\neq y$, there is a unique rectifiable arc with endpoints $x$ and $y$, then $X$ is a geodesic space.
\end{lemma}
\begin{proof}
Let $x,y\in X$, $x\neq y$.
Since $X$ is a length space, Corollary~\ref{T17} yields existence of a rectifiable arc $\Gamma$ with endpoints $x$ and $y$. It remains to show that $\ell(\Gamma)=d(x,y)$.
If by contrary, $\ell(\Gamma)>d(x,y)$, then Corollary~\ref{T17} yields another arc $\Gamma'$ connecting $x$ to $y$ and such that $\ell(\Gamma)>\ell(\Gamma')\geq d(x,y)$, so clearly $\Gamma\neq\Gamma'$. This, however, contradicts the uniqueness of a rectifiable arc connecting $x$ and $y$.
\end{proof}

The next result provides several characterizations of a metric tree.
\begin{theorem}
\label{T28}
Let $X$ be a metric space. Then the following conditions are equivalent.
\begin{itemize}
\item[(a)] $X$ is a metric tree.
\item[(b)] $X$ is a length space and for any $x,y\in X$, $x\neq y$, there is a unique arc with endpoints $x$ and $y$.
\item[(c)] $X$ is a length space and for any $x,y\in X$, $x\neq y$, there is a unique rectifiable arc with endpoints $x$ and $y$.
\item[(d)] $X$ is a length space and there is no one-to-one Lipschitz map $\alpha:\bbbs^1\to X$.
\item[(e)] $X$ is a length space and it has the following property:
for any Lipschitz maps
$\alpha:\bbbs^1\to X$ and $\pi:X\to\bbbr^2$, such that $\pi\circ \alpha$ maps $\bbbs^1$ to $\bbbs^1\subset\bbbr^2$, the map $\pi\circ \alpha:\bbbs^1\to\bbbs^1$ is not homotopic to the identity map.
\end{itemize}
\end{theorem}
\begin{proof}
It suffices to prove implications
(a)$\Leftrightarrow$(b), (a)$\Rightarrow$(c)$\Rightarrow$(b),  (c)$\Leftrightarrow$(d)$\Leftarrow$(e) and (a)$\Rightarrow$(e)

The implication (a)$\Rightarrow$(b) is obvious. 

To prove (b)$\Rightarrow$(a) we only need to show that $X$ is a geodesic space. Since $X$ is a length space, Corollary~\ref{T17} yields existence of a rectifiable arc connecting $x$ to $y$, $x\neq y$. Clearly, by the assumptions of (b), it must be a unique rectifiable arc connecting $x$ and $y$, so $X$ is a geodesic space by Lemma~\ref{T27}.

The implication (a)$\Rightarrow$(c) is obvious. Indeed, since $X$ is a geodesic space, there is a geodesic connecting $x$ and $y$, $x\neq y$, and by Lemma~\ref{T14}, the image of the geodesic is a rectifiable arc. Since it is the unique arc connecting $x$ and $y$, it is also the unique rectifiable arc connecting $x$ and $y$.

(c)$\Rightarrow$(b) According to Lemma~\ref{T27}, $X$ is a geodesic space. Thus, if $x,y\in X$, $x\neq y$, the image $\Gamma$ of the geodesic connecting $x$ to $y$ is the unique rectifiable arc with endpoints $x$ and $y$. It remains to show that $\Gamma$ is the unique arc with the endpoints $x$ and $y$.

To this end it suffices to show that any arc $A$ connecting $x$ to $y$ is contained in $\Gamma$, $A\subset\Gamma$, because it clearly implies that $A=\Gamma$.

Suppose by way of contradiction that there is an arc $A$ with endpoints $x$ and $y$, and such that $A\setminus\Gamma\neq\varnothing$.

Let $\alpha:[0,1]\to A$ be a homeomorphism, $\alpha(0)=x$, $\alpha(1)=y$. Since $A$ is not contained in $\Gamma$, there is $t_0\in (0,1)$, such that $\alpha(t_0)\not\in\Gamma$.
We can find $a\in [0,t_0)$ and $b\in (t_0,1]$ such that
$$
\tilde{x}:=\alpha(a)\in\Gamma,
\quad
\tilde{y}:=\alpha(b)\in\Gamma,
\quad
\text{and}
\quad
\alpha(a,b)\cap\Gamma=\varnothing.
$$
Denote the sub-arc of $\Gamma$ with endpoints $\tilde{x}$ and $\tilde{y}$ by $\tilde{\Gamma}$.
Choose
$$
a<t_{-N}<t_{-N+1}<\ldots<t_0<\ldots<t_N<b,
$$
such that
\begin{equation}
\label{eq15}
d(\alpha(t_i),\alpha(t_{i+1}))<\dist(\alpha(t_i),\tilde{\Gamma})
\quad
\text{for $i=-N,\ldots,N-1$,}
\end{equation}
and
\begin{equation}
\label{eq16}
d(\tilde{x},\alpha(t_{-N}))+d(\alpha(t_N),\tilde{y})<d(\tilde{x},\tilde{y})=\ell(\tilde{\Gamma}).
\end{equation}
Next, we connect consecutive points 
$\tilde{x}$, $\alpha(t_{-N})$, $\alpha(t_{-N+1})$, \ldots, $\alpha(t_N)$, $\tilde{y}$, by geodesics (we already proved that $X$ is a geodesic space).

It follows from \eqref{eq15} that the geodesics connecting $\alpha(t_i)$ to $\alpha(t_{i+1})$ do not intersect with $\tilde{\Gamma}$. Since by \eqref{eq16},  the sum of lengths of geodesics connecting $\tilde{x}$ to $\alpha(t_{-N})$ and $\alpha(t_N)$ to $\tilde{y}$ is less than $\ell(\tilde{\Gamma})$, $\tilde{\Gamma}$ is not contained in the image of these two geodesics. 

Let $\eta$ be a rectifiable curve connecting $\tilde{x}$ to $\tilde{y}$ obtained by concatenation of the geodesics constructed above. Let $E$ be the image of $\eta$. As we observed above, $\tilde{\Gamma}$ is not contained in~$E$. 

According to Lemma~\ref{T16}, there is a rectifiable arc inside $E$ connecting $\tilde{x}$ to $\tilde{y}$. Since $\tilde{\Gamma}\setminus E\neq\varnothing$, this arc must be different than $\tilde{\Gamma}$ and we arrived to a contradiction with the uniqueness of a rectifiable arc connecting $\tilde{x}$ and $\tilde{y}$. The proof is complete. 

The implication (c)$\Rightarrow$(d) is obvious:  assume, by contradiction, that $X$  is a length space and $\alpha:\bbbs^1\to X$ is one-to-one and Lipschitz. Then the images of the upper and the lower semicircles are distinct rectifiable arcs connecting the same endpoints.

Similarly, the implication  (d)$\Rightarrow$(c) follows by contradiction from Lemma~\ref{T29}.

The implication (e)$\Rightarrow$(d) follows by contradiction from Proposition~\ref{T23}.

Finally, we prove the implication (a)$\Rightarrow$(e). Suppose that $X$ is a metric tree and Lipschitz maps $\alpha:\bbbs^1\to X$ and $\pi:X\to\bbbr^2$ are such that $\pi\circ \alpha$ maps $\bbbs^1$ to $\bbbs^1\subset\bbbr^2$. It is easy to see that the image $\alpha(\bbbs^1)\subset X$ is also a metric tree. 
Indeed, any two points $x,y\in\alpha(\bbbs^1)$, $x\neq y$, can be connected by a rectifiable curve and hence by a rectifiable arc inside $\alpha(\bbbs^1)$ (see Lemma~\ref{T16}). Clearly, this is a unique arc connecting $x$ and $y$, and its length equals $d(x,y)$, because $X$ is a metric tree. Hence $\alpha(\bbbs^1)$ is a metric tree by definition.
Therefore, $\alpha(\bbbs^1)$ is contractible by Lemma~\ref{T30}, and hence the map $\pi\circ\alpha:\bbbs^1\to\bbbs^1$ is homotopic to a constant map, so it is not homotopic to the identity map.
\end{proof}

The next lemma is, in fact, an easy exercise, but we include it for the sake of completeness.

\begin{lemma}\label{L12}
Assume $(X,d)$ is a metric tree and $T\subset X$ is a closed, non-empty metric tree. Then there is a 1-Lipschitz retraction $r:X\to T$, i.e., $r|_T=\id_T$, $d(r(x),r(y))\leq d(x,y)$. Moreover, for all $x\in X$ we have $d(x,r(x))=\dist(x,T)$.
\end{lemma}
\begin{proof}
To simplify the notation, whenever $x,y\in X$, we denote the (unique) arc connecting $x$ and $y$ in $X$ by $\la x,y\ra$. We can clearly assume that $T\varsubsetneq X$.

To prove the existence of the retraction we first establish the following auxiliary facts:

For any $x\in X$ there is a unique $t_x\in T$ such that
\begin{enumerate}[1)]
\item $\la t_x,x\ra\cap T=\{t_x\}$,
\item for any $y\in T$ we have $t_x\in \la y,x\ra$,
\item $\dist(x,T)=d(x,t_x)$.
\end{enumerate}
If $x\in T$, obviously $t_x=x$ is the only point satisfying 1), 2), and 3), so assume that $x\not \in T$.

To prove the existence of $t_x$ satisfying 1), pick  any $y\in T$ and let $\gamma:[0,1]\to X$ parameterize the arc $\la y,x\ra$ from $y$ to $x$. The set $\la y,x\ra\cap T$ is closed, so if $s=\sup\{s'\in[0,1]~:~\gamma(s')\in T\}$, then $t_x=\gamma(s)$ is in $T$, and $\gamma([s,1])= \la  t_x,x\ra$ intersects with $T$ only at~$t_x$.

Uniqueness of $t_x$ follows immediately from uniqueness of arcs in $X$: assume that for some $x\in X$ there are two distinct $t_x$ and $t_x'$ satisfying 1). Then there are two distinct arcs connecting $x$ to $t_x$: $\la x, t_x \ra$ and $\la x,t_x'\ra \cup \la t_x',t_x\ra$, which is a contradiction.
Note that $\la x,t_x'\ra \cup \la t_x',t_x\ra$ is an arc, because $\la x,t_x'\ra$ meets $T$ at $t_x'$ only and $\la t_x',t_x\ra\subset T$, since  $T$ is a metric tree.

Now, 2) follows from the construction of $t_x$: we constructed it as a point on an arc $\la y,x\ra$ for any $y\in T$ and proved that any choice of $y$ yields the same (unique) $t_x$; 3) follows immediately from 2).

Finally, we set $r(x)=t_x$ and note that $d(r(x),r(y))\leq d(x,y)$ is an immediate consequence of the easy observation that for any $x,y\in X$ we either have $r(x),r(y)\in \la x,y\ra$ or $r(x)=r(y)$. Also, by 3), $d(x,r(x))=\dist(x,T)$.
\end{proof}

\section{Derivatives of Lipschitz mappings}
\label{S5}
We assume that the reader is familiar with basic results about differentiability of Lipschitz mappings, like for example Rademacher's theorem. Lemmata~\ref{T22} and \ref{LR} below are well known. The other two lemmata, while possibly known to specialists, seem to be missing in the literature, and the only relevant reference we are aware of is \cite{EsmayliH}.

For a proof of the following result see e.g., \cite[Theorem~3.8]{EvansG}.
\begin{lemma}
\label{T22}
If $f:\bbbb^n\to\bbbr^n$ is Lipschitz continuous, then the measure of the image of $f$ is bounded by the integral of the Jacobian:
$$
|f(\bbbb^n)|\leq\int_{\bbbb^n}|\det Df(x)|\, dx.
$$
\end{lemma}

\begin{lemma}
\label{T11}
Suppose that the mappings $f:\Omega\to\bbbr^m$ and $g:\Omega\to\bbbr^k$ are locally Lipschitz continuous, where $\Omega\subset\bbbr^n$ is open. 
If there is a constant $L>0$ such that for every rectifiable curve $\gamma:[a,b]\to\Omega$ we have
$$
\ell(g\circ\gamma)\leq L\ell(f\circ\gamma),
$$
then for almost every $x\in\Omega$, we have $|Dg(x)v|\leq L|Df(x)v|$ for all $v\in\bbbs^{n-1}$, and hence $\rank Dg\leq\rank Df$ almost everywhere in $\Omega$.
\end{lemma}
\begin{proof}
Since the problem is local in nature, we may assume that $\Omega=\bbbb^n$, and that the mappings $f$ and $g$ are Lipschitz continuous.

Let $\{ v_i\}_{i=1}^\infty\subset\bbbs^{n-1}$ be a countable and dense subset. It suffices to prove that for almost all $x\in\bbbb^n$ we have
\begin{equation}
\label{eq10}
|Dg(x)v_i|\leq L|Df(x)v_i|
\quad
\text{for all $i=1,2,\ldots$}
\end{equation}
Indeed, by a density argument it will imply that 
$|Dg(x)v|\leq L|Df(x)v|$
for all $v\in\bbbs^{n-1}$.

Fix $i\in\bbbn$. It suffices to prove that 
\begin{equation}
\label{eq11}
|Dg(x)v_i|\leq L |Df(x)v_i|
\quad
\text{for almost all $x\in\bbbb^n$.}
\end{equation}
Indeed, since the union of countably many sets of measure zero has measure zero, \eqref{eq10} will follow.

For almost every line $\ell$ parallel to $v_i$, both functions $f$ and $g$ are differentiable at almost all points of $\ell$. Fix such a line $\ell$ and for $z\in\ell$ define
$\gamma_z(t)=z+tv_i$. Let $I\subset\bbbr$ be the open interval consisting of all $t$ such that $\gamma_z(t)\in\bbbb^n$.

Since the functions $g\circ\gamma_z$ and $f\circ\gamma_z$ are Lipschitz continuous on $I$, for $s,t\in I$ we have
$$
\int_s^t\Big|\frac{d}{d\tau}(g\circ\gamma_z)(\tau)\Big|\, d\tau=
\ell\big((g\circ\gamma_z)|_{[s,t]}\big)\leq
L \ell\big((f\circ\gamma_z)|_{[s,t]}\big)=
L\int_s^t\Big|\frac{d}{d\tau}(f\circ\gamma_z)(\tau)\Big|\, d\tau.
$$
Since the functions $f$ and $g$ are differentiable at almost all points of $\ell$, the chain rule yields
$$
\int_s^t |Dg(z+\tau v_i)v_i|\, d\tau\leq
L \int_s^t |Df(z+\tau v_i)v_i|\, d\tau.
$$
Now it follows from the Lebesgue differentiation theorem that for almost all $s\in I$ we have 
$$
|Dg(z+sv_i)v_i|\leq L|Df(z+sv_i)v_i|.
$$
We proved that inequality \eqref{eq11} is true for almost all lines $\ell$ parallel to $v_i$ and for almost all $x\in\ell\cap\bbbb^n$. Therefore it follows from Fubini's theorem that \eqref{eq11} is true for almost all $x\in\bbbb^n$.
\end{proof}
The next lemma is a well known consequence of Brouwer's theorem (c.f. \cite[Lemma~7.23]{Rudin}):
\begin{lemma}
\label{LR}
Assume $F:\overbar{\bbbb}^n(0,\rho)\to \bbbr^n$ is  continuous and satisfies ${|F(x)-x|<\rho/2}$ whenever $|x|=\rho$. Then $\overbar{\bbbb}^n(0,\rho/2)\subset F(\overbar{\bbbb}^n(0,\rho))$. 
\end{lemma}
\begin{proof}
By contradiction, if there exists $a\in \overbar{\bbbb}^n(0,\rho/2)\setminus F(\overbar{\bbbb}^n(0,\rho))$, then 
$$
G(x)=\rho\, \frac{a-F(x)}{|a-F(x)|}
$$ 
is continuous and maps $\overbar{\bbbb}^n(0,\rho)$ to itself without a fixed point, which contradicts Brou\-wer's theorem. Indeed, any fixed point would have to lie on the sphere $\bbbs^{n-1}(0,\rho)$, but if $|x|=\rho$, then 
$$
x\cdot (a-F(x))=x\cdot(a+x-F(x))-|x|^2\leq |x|(|a|+|x-F(x)|)-\rho^2<\rho^2-\rho^2=0,
$$
so $x\cdot G(x)<0$ and hence $x\neq G(x)$.
\end{proof}
\begin{lemma}
\label{T21}
Let $\Omega\subset\bbbr^n$, $U\subset\bbbr^m$ be open, and let $g:\Omega\to U$, $f:U\to\bbbr^k$ be Lipschitz continuous. If $\rank Df\leq r$ almost everywhere in $U$, then $\rank D(f\circ g)\leq r$ almost everywhere in $\Omega$.
\end{lemma}
If $f\in C^1$ satisfies $\rank Df\leq r$ everywhere in $U$, then the lemma is an obvious consequence of the chain rule. However, if $f$ is Lipschitz continuous only, the lemma is very far from being obvious, because the image of $g$ might be contained in the set where $f$ is not differentiable and the chain rule cannot be applied. 

Essentially, Lemma \ref{T21} is Proposition~3.16 in \cite{EsmayliH}. There, however, the result is considered and proved for a much more general concept of the metric derivative, so while the proof in \cite{EsmayliH} quickly reduces to the Euclidean setting, we present it for the reader's convenience, avoiding the (unnecessary here) general metric space setting.

\begin{proof}
Assume, by the way of contradiction, that $\rank D(f\circ g)\geq r+1$ on a set of positive measure in $\Omega$. Then, we can find a positive measure set $E\subset\Omega$ and an $(r+1)\times(r+1)$ minor of $D(f\circ g)$ such that this minor is non-zero in $E$. Without loss of generality we may assume that the minor corresponds to the first $(r+1)$ coordinates, both in the domain and in the image, so that
\begin{equation}
\label{eP1}
\det \left(\frac{\partial (f_i\circ g)}{\partial x_j}(x)\right)_{1\leq i,j\leq r+1}\neq 0 \qquad \text{ for }x\in E.
\end{equation}

In what follows, we shall restrict our attention to these $(r+1)$ coordinates. First, we set $F=(f_1,\ldots,f_{r+1}):U\to\bbbr^{r+1}$. Showing that $\rank DF=r+1$ on a set of positive measure would imply $\rank Df=r+1$ on a set of positive measure, leading to a contradiction.

Next, we pick $x_o\in E$ such that $g$ is differentiable at $x_o$; without loss of generality we may assume that $x_o=0$ and $g(0)=0$, and likewise $F(0)=0$.

Then, we restrict $g$ to $\Omega'=\Omega\cap H$, where $H$ is the linear subspace spanned by the first $(r+1)$ coordinates (note that $x_o=0\in \Omega'$), setting $G=g|_{\Omega'}:\Omega'\to U$.
Then
\begin{equation}\label{eP2}
\det D(F\circ G)(0)\neq 0.
\end{equation}
Since $F$ is Lipschitz and $G$ differentiable at $0$, \eqref{eP2} implies that $\rank DG(0)=r+1$. Indeed, if $L$ is the Lipschitz constant of $F$, then all the directional derivatives of $F\circ G$ at~$0$ satisfy $|D_v(F\circ G)(0)|\leq  L|D_v G(0)|$, and thus 
$$
r+1=\rank D(F\circ G)(0)\leq \rank DG(0)\leq r+1.
$$

To simplify the setting, we post-compose $G$ with a linear isomorphism of $\bbbr^m$ to have for all $v\in \bbbr^{r+1}$,
$$
v= (v_1,\ldots,v_{r+1})\xmapsto{~~DG(0)~~}(v_1,\ldots,v_{r+1},0,\ldots,0)=(v,0)\in \bbbr^{r+1}\times \bbbr^{m-r-1}=\bbbr^m 
$$
and since $D(F\circ G)(0):\bbbr^{r+1}\to\bbbr^{r+1}$ is an isomorphism, we may post-compose $F$ with a linear isomorphism of $\bbbr^{r+1}$ to have $D(F\circ G)(0)=\id$.

Then $G(x)=(x,0)+o(|x|)$ and $(F\circ G)(x)=x+o(|x|)$. Thus,  we may find $\rho>0$ such that $\overbar{\bbbb}^{r+1}(0,\rho)\times \overbar{\bbbb}^{m-r-1}(0,\rho/6L)\subset U$ and whenever $|x|=\rho$, 
$$
|(F\circ G)(x)-x|<\frac{\rho}{6} 
\quad
\text{and}
\quad
|G(x)-(x,0)|<\frac{\rho}{6L}.
$$
For any $y$ with $|y|<\rho/6L$ and $|x|=\rho$ we have
\begin{equation*}
    \begin{split}
       |F(x,y)-x|&\leq |F(x,y)-F(x,0)|+|F(x,0)-F(G(x))|+|F(G(x))-x| \\
       &\leq L|y|+L|(x,0)-G(x)|+\frac{\rho}{6}<\frac{\rho}{2}.
    \end{split}
\end{equation*}
This, together with Lemma~\ref{LR}, implies that for every $y$ with $|y|<{\rho}/{6L}$
$$
\overbar{\bbbb}^{r+1}(0,\rho/2)\subset F(\overbar{\bbbb}^{r+1}(0,\rho)\times\{y\}),
$$
in particular, the Lebesgue measure of $F(\overbar{\bbbb}^{r+1}(0,\rho)\times\{y\})$ is positive. 
Now, by Fubini's theorem, $F$ is differentiable a.e. on $\bbbb^{r+1}(0,\rho)\times \{y\}$ for almost every $y$ with $|y|<\rho/6L$, so for any such $y$, by Lemma \ref{T22}, 
$$
\rank D(F|_{\overbar{\bbbb}^{r+1}(0,\rho)\times\{y\}})=r+1
$$ 
on a positive measure subset of $\overbar{\bbbb}^{r+1}(0,\rho)\times\{y\}$. This, again by Fubini's theorem, implies $\rank DF\geq r+1$ on a positive measure subset of $\overbar{\bbbb}^{r+1}(0,\rho)\times \overbar{\bbbb}^{m-r-1}(0,\rho/6L)\subset U$, which gives the desired contradiction.
\end{proof}

\section{Factorization of Lipschitz mappings}
\label{S2}

The content of this section is based on \cite{EsmayliH,WengerY}. While the constructions in \cite{EsmayliH,WengerY} were carried out in a general framework of Lipschitz mappings between metric spaces, we specify the construction here to the case of locally Lipschitz mappings $f:\Omega\to\bbbr^m$, $\Omega\subset\bbbr^n$.

\subsection{Canonical factorization}
Assume that $\Omega\subset \bbbr^n$ is a domain and let $f:\Omega\to\bbbr^m$ be a locally Lipschitz mapping.

We say that the mapping $f$ \emph{factors} through a metric space $X$, if there is a locally Lipschitz map $\psi:\Omega\to X$ and a $1$-Lipschitz map $\phi:X\to\bbbr^m$ such that $f=\phi\circ \psi$.

Next, we describe a particular construction of a factorization of $f$.

We define a quasimetric in $\Omega$ by 
$$ 
d_f(x,y)=\inf \{\ell (f\circ\gamma)\}, 
$$
where the infimum is taken over all rectifiable curves $\gamma:[0,1]\to \Omega$ such that $\gamma(0)=x$ and $\gamma(1)=y$. 

Clearly, $d_f(x,y)=d_f(y,x)$ and $d_f$ satisfies the triangle inequality, but it is a quasimetric, since it may happen that $d_f(x,y)=0$ for some $x\neq y$.

It is easy to see that 
\begin{equation}
\label{eq3}
    |f(x)-f(y)|\leq d_f(x,y)
\end{equation}
and that for any compact set $K\subset \Omega$ there is a constant $L_K>0$ such that 
\begin{equation}
\label{eq4}
d_f(x,y)\leq L_K |x-y| \quad \text{for all } x,y\in K.
\end{equation}

Indeed, if $\eps<\dist(K,\partial\Omega)$, then the $\eps$-neighborhood of $K$
\begin{equation}
\label{eq17}
V_\eps:=\bigcup_{x\in K}\bbbb^n(x,\eps)
\end{equation}
satisfies $\overbar{V}_\eps\subset\Omega$. Take a finite sub-cover $\{ B_i\}_{i=1}^N$ of $K$ from the covering \eqref{eq17}. Let $\gamma_i:[0,1]\to\Omega$, $i=1,2,\ldots,N-1$, be rectifiable curves connecting the centers of the consecutive balls $B_i$ and $B_{i+1}$, and let $\Gamma_i=\gamma_i([0,1])$. 
The set
$$
\widetilde{K}:=\overbar{V}_\eps\cup\bigcup_{i=1}^{N-1} \Gamma_i\subset\Omega
$$
is compact. Let $\Lambda$ be the Lipschitz constant of $f|_{\widetilde{K}}$ (see Lemma \ref{T3}).

If $M$ equals $2\eps$ plus the sum of lengths of curves~$\gamma_i$, then any points $x,y\in K$ can be connected by a curve in $\widetilde{K}$ of length at most~$M$.

Take any points $x,y\in K$. If $|x-y|<\eps$, then the segment $[x,y]$ is contained in $V_\eps$ and hence
$$
d_f(x,y)\leq\ell(f([x,y]))\leq \Lambda |x-y|.
$$
If $|x-y|\geq \eps$ and $\gamma$ is a curve of length at most $M$ connecting $x$ and $y$ inside $\widetilde{K}$, then 
$$
d_f(x,y)\leq\ell(f\circ\gamma)\leq\Lambda M\leq\Lambda M\eps^{-1}|x-y|.
$$
Thus, \eqref{eq4} is satisfied with $L_K:=\max\{\Lambda,\Lambda M\eps^{-1}\}$.

Inequality \eqref{eq3} yields 
\begin{equation}\label{eq5}
d_f(x,y)=0\quad\Rightarrow\quad f(x)=f(y).
\end{equation}
However, in general, the converse implication is false.

We define an equivalence relation in $\Omega$ by
$$
x \sim y\quad\text{if and only if}\quad d_f(x,y)=0
$$
and then we define $Z_f:=\Omega\,/\sim$ with the quotient metric
\begin{equation}\label{eq6}
    d_f([x],[y]):=d_f(x,y),
\end{equation}
where $[x]=\{x'\in\Omega ~~:~~x\sim x'\}$. It is easy to check that \eqref{eq6} is well defined, i.e., if $x\sim x'$ and $y\sim y'$, then $d_f(x,y)=d_f(x',y')$.

The next result is an easy exercise.
\begin{lemma}
\label{T5}
$(Z_f,d_f)$ is a metric space.
\end{lemma}
Now, we define mappings 
$$
\Omega \xrightarrow{\psi} Z_f \xrightarrow{\phi}\bbbr^m 
\quad
\text{by} 
\quad
\psi(x)=[x], \,\, \phi([x])=f(x),
\quad
\text{so}
\quad
f=\phi\circ\psi.
$$
The mapping $\phi$ is well defined, because by \eqref{eq5}, if $[x]=[x']$, i.e., $x\sim x'$, then $f(x)=f(x')$.
\begin{lemma}
\label{T6}
The mapping $\psi:\Omega \to Z_f$ is locally Lipschitz and the mapping $\phi:Z_f\to \bbbr^m$ is $1$-Lipschitz. 
Hence, $f:\Omega\to\bbbr^m$ factors through $Z_f$, $f=\phi\circ\psi$.
\end{lemma}
\begin{proof}
The mapping $\psi$ is locally Lipschitz, because according to \eqref{eq4}, for any compact $K\subset\Omega$ we have 
$$
d_f(\psi(x),\psi(y))=d_f([x],[y])=d_f(x,y)\leq L_K|x-y| \text{ for all }x,y\in K.
$$
Also, the mapping $\phi$ is $1$-Lipschitz, because \eqref{eq3} yields 
$$
|\phi([x])-\phi([y])|=|f(x)-f(y)|\leq d_f(x,y)=d_f([x],[y]).
$$
\end{proof}
Composing with a $1$-Lipschitz mapping cannot increase the length of a curve, thus for any curve $\alpha:[0,1]\to Z_f$ we have $\ell(\phi\circ\alpha)\leq \ell(\alpha)$.
\begin{lemma}
\label{T7}
If $\gamma:[0,1]\to  \Omega$
 is a rectifiable curve and $\alpha=\psi\circ \gamma:[0,1]\to Z_f$, then $\ell(\alpha)=\ell(\phi\circ\alpha)$.
\end{lemma}
For a proof, see \cite[Lemma 6.4]{EsmayliH}. In other words, $\phi$ preserves lengths of curves in $Z_f$ that are images of rectifiable curves in $\Omega$.
\begin{lemma}
\label{T8}
$(Z_f,d_f)$ is a length space.
\end{lemma}
For a proof, see \cite[Corollary 6.5]{EsmayliH}.
\begin{lemma}
\label{T9}
Let $\alpha:\bbbs^1\to Z_f$ be a Lipschitz curve. Then there is a sequence of Lipschitz curves $\gamma_k:\bbbs^1\to\Omega$ such that $\alpha_k=\psi\circ\gamma_k:\bbbs^1\to Z_f$ converge uniformly to $\alpha$, and $\ell(\alpha_k)\to \ell(\alpha)$.
\end{lemma}
For a proof, see \cite[Lemma 6.6]{EsmayliH}.

We will discuss now a situation when the metric space $Z_f$ is a metric tree.

\subsection{Factorization through metric trees}
The next result is similar to Theorem~1.9 in \cite{EsmayliH}. While \cite{EsmayliH} deals with a more general factorization of Lipschitz maps into metric spaces, the mappings considered in \cite{EsmayliH} are defined on the compact cube $[0,1]^n$. The fact that $\Omega$ is not compact causes some additional problems. At the same time, our argument is simpler than the one used in \cite{EsmayliH} since we avoid the use of metric area formula and that of differential forms. However, the overall idea of the proof remains the same.

\begin{proposition}
\label{T10}
Let $\Omega\subset\bbbr^n$ be a simply connected domain and let $f:\Omega\to\bbbr^m$ be a locally Lipschitz map. Then $Z_f$ is a metric tree if and only if $\rank Df\leq 1$ almost everywhere. 
\end{proposition}
\begin{proof}
If $Z_f$ is a metric tree, then $\rank Df\leq 1$ almost everywhere by \cite[Theorem~5.6]{EsmayliH}. 

While we do not provide details of the proof of this implication, let us emphasize that it is not needed in the proof of Theorem~\ref{T26}. We only need the other implication: that if $\rank Df\leq 1$, then $Z_f$ is a metric tree, and we prove it carefully below.

Suppose by way of contradiction that $Z_f$ is not a metric tree. Since $Z_f$ is a length space, Theorem~\ref{T28}(e) yields Lipschitz maps $\tilde{\alpha}:\bbbs^1\to Z_f$ and $\tilde{\pi}:Z_f\to\bbbr^2$, such that
$\tilde{\pi}\circ\tilde{\alpha}$ maps $\bbbs^1$ to $\bbbs^1\subset\bbbr^2$ and $\tilde{\pi}\circ\tilde{\alpha}:\bbbs^1\to\bbbs^1$ is homotopic to the identity map.

Let $H:\bbbr^2\to\bbbr^2$ be a Lipschitz map such that $H(x)=x/|x|$ for $|x|\geq 1/2$.

Let $\alpha_k=\psi\circ\gamma_k$ be the Lipschitz approximation of $\tilde{\alpha}$ from Lemma~\ref{T9}. Since $\alpha_k\to\tilde{\alpha}$ uniformly, $\tilde{\pi}\circ\alpha_k\to\tilde{\pi}\circ\tilde{\alpha}$ uniformly, and hence $|\tilde{\pi}\circ\alpha_k|>1/2$ for all sufficiently large $k$. 
Thus,
$$
H\circ\tilde{\pi}\circ\alpha_k=\frac{\tilde{\pi}\circ\alpha_k}{|\tilde{\pi}\circ\alpha_k|}:\bbbs^1\to\bbbs^1,
\quad
H\circ\tilde{\pi}\circ\alpha_k\to \tilde{\pi}\circ\tilde{\alpha}
\text{ uniformly.}
$$
Since $\tilde{\pi}\circ\tilde{\alpha}$ is homotopic to the identity, 
$H\circ\tilde{\pi}\circ\alpha_k$ is homotopic to the identity for sufficiently large $k$.
Therefore, if $\alpha=\alpha_k$ for sufficiently large $k$ and $\pi=H\circ\tilde{\pi}$, then 
\begin{enumerate}
\item $\alpha:\bbbs^1\to Z_f$ is of the form $\alpha=\psi\circ\gamma$, where $\gamma:\bbbs^1\to\Omega$ is Lipschitz,
\item $\pi:Z_f\to\bbbr^2$ is Lipschitz,
\item $\pi\circ\alpha:\bbbs^1\to\bbbs^1$ is homotopic to the identity.
\end{enumerate}

Since $\Omega$ is simply connected, $\gamma$ admits a continuous extension $\hat{g}:\overbar{\bbbb}^2\to\Omega$. Using standard approximation, we may then improve $\hat{g}$ to a Lipschitz map $g:\overbar{\bbbb}^2\to\Omega$ such that $g|_{\partial\bbbb^2}=\gamma$. Then
$\pi\circ\psi\circ g:\overbar{\bbbb}^2\to\bbbr^2$ is a Lipschitz extension of $\pi\circ\psi\circ\gamma=\pi\circ\alpha:\bbbs^1\to\bbbs^1$.
The mapping $\pi\circ\alpha$ is homotopic to the identity, so it follows that 
$\bbbb^2\subset(\pi\circ\psi\circ g)(\bbbr^2)$. Since the image of the Lipschitz map $\pi\circ\psi\circ g$ has positive area, Lemma~\ref{T22} implies that its Jacobian must be non-zero on a set of positive measure and we will arrive to a contradiction as soon as we prove the following claim:
\begin{equation}
\label{eq12}
\det D(\pi\circ\psi\circ g)=0
\quad
\text{almost everywhere in $\bbbb^2$.}
\end{equation}
To prove this, we will use Lemma~\ref{T11}, so we need to consider rectifiable curves in $\bbbb^2$. Let $\eta:[a,b]\to\bbbb^2$ be a rectifiable curve, then $g\circ\eta:[a,b]\to\Omega$ is rectifiable and Lemma~\ref{T7} yields
$$
\ell(\psi\circ g\circ\eta)=\ell(\phi\circ\psi\circ g\circ\eta)=\ell(f\circ g\circ\eta).
$$
If $L$ is the Lipschitz constant of $\pi$, then
$$
\pi\circ\psi\circ g:\bbbb^2\to\bbbr^2
\quad
\text{and}
\quad
f\circ g:\bbbb^2\to\bbbr^m
$$
are Lipschitz continuous, and 
$$
\ell(\pi\circ\psi\circ g\circ\eta)\leq L\,\ell(\psi\circ g\circ\eta)=
L\,\ell(f\circ g\circ\eta),
$$
so Lemma~\ref{T11}, Lemma~\ref{T21} and the fact that $\rank Df\leq 1$ a.e. yield that
$$
\rank D(\pi\circ\psi\circ g)\leq \rank D(f\circ g)\leq 1
\quad
\text{a.e.}
$$
This proves \eqref{eq12}. The proof is complete.
\end{proof}

\section{Proof of Theorem \ref{T26}}
\label{S3}
This section is entirely devoted to the proof of Theorem \ref{T26}. We will need the following lemma:

\begin{lemma}
\label{L:eps}
Assume $\Omega\subset \bbbr^n$ is a domain. Fix $x_o\in \Omega$ and for $\eps>0$ let $\Omega_\eps$ be the connected component of the open set 
$\{x\in \Omega~~:~~|x-x_o|<\eps^{-1},~ \dist(x,\partial\Omega)>\eps\}$
containing the point~$x_o$. Then the family  $\{\Omega_\eps\}_{\eps>0}$ has the following properties:
\begin{itemize}
    \item[(a)] the sets $\overbar{\Omega}_\eps$ are compact and connected in $\Omega$,
    \item[(b)] whenever $\eps<\eps'$, we have $\overbar{\Omega}_{\eps'}\subset \Omega_\eps$,
    \item[(c)] $\bigcup_{\eps>0} \Omega_\eps=\Omega$,
    \item[(d)] for any compact $K\subset\Omega$ there is $\eps_K>0$ such that for any $\eps\in (0,\eps_K)$ we have $K\subset \Omega_\eps$.
    \item[(e)] $\dist(\Omega_\eps,\partial \Omega)\geq \eps$ and $\diam \Omega_\eps\leq 2\eps^{-1}$.
\end{itemize}
\end{lemma}
Checking the properties (a) through (e) of $\Omega_\eps$ is straightforward.
\qed

By the construction given in Section \ref{S2}, the mapping $f$ factorizes:
$$
\Omega\xrightarrow{\psi}Z_f\xrightarrow{\phi}\bbbr^m, \quad f=\phi\circ\psi,
$$
where $\psi$ is locally Lipschitz, $\phi$ is $1$-Lipschitz, and, according to Proposition \ref{T10}, $(Z_f, d_f)$ is a metric tree.

Let us next fix $\eps>0$ and let $\Omega_\eps$ be as in Lemma \ref{L:eps}. 

Since $\overbar{\Omega}_\eps$ is compact and connected in $\Omega$, the set $\psi(\overbar{\Omega}_\eps)$ is compact and connected in $Z_f$.

The outline of the rest of the proof was given in the Introduction, here we state it in some more detail:
\begin{enumerate}[1)]
\item We find a finite tree $T\subset Z_f$, approximating $\psi(\overbar{\Omega}_\eps)$, so that the retraction ${r:Z_f\to T}$ (see Lemma \ref{L12}) is close to the identity on $\psi(\overbar{\Omega}_\eps)$.
\item We embed $T$ in $\bbbr^E$, where $E$ is the number of edges in $T$, $w:T\to\bbbr^E$,
so that the edges in $w(T)$ are mutually orthogonal.
\item For some small $\delta>0$ depending on $\eps$, we construct a $C^\infty$-smooth mapping $\rho_\eps$ from a $\delta$-neighborhood $V_\delta$ of $w(T)$ in $\bbbr^E$ onto $w(T)$; $\rho_\eps|_{w(T)}$ need not be the identity, but it is close to the identity. Clearly, $\rank D\rho_\eps\leq 1$.
\item We approximate  $g=w\circ r\circ\psi:\Omega\to \bbbr^E$ by a $C^\infty$-smooth $g_\eps:\Omega\to\bbbr^E$, ${\sup_{\Omega}|g-g_\eps|< \delta}$.
Clearly, the image of $g_\eps$ lies in $V_\delta$.
\item We extend $\phi\circ w^{-1}:w(T)\to\bbbr^m$ to a Lipschitz map $\tilde{\phi}:\bbbr^E\to\bbbr^m$ and approximate $\tilde{\phi}$ with a smooth $\phi_{\eps}:\bbbr^E\to\bbbr^m$.
\end{enumerate}

Then $f_\eps=\phi_\eps\circ\rho_\eps\circ g_\eps$ is a smooth almost-uniform approximation of $f$.
Since $\rank D\rho_\eps\leq 1$, it follows that $\rank Df_\eps\leq 1$.

\noindent
{\bf Construction of a finite sub-tree $T\subset Z_f$ approximating $\psi(\overbar{\Omega}_\eps)$.}

Let $A\subset \psi(\overbar{\Omega}_\eps)$ be a finite $\eps$-net in $\psi(\overbar{\Omega}_\eps)$. The tree $T$ consists of all the geodesic arcs in $Z_f$ connecting the points of the $\eps$-net $A$. It is easy to see that if $k$ is the number of points in $A$ (and $k>1$), then $T$ has at most $2k-2$ vertices and $2k-3$ edges. Indeed, for $k=2$ we have 2 vertices and 1 edge. Suppose we already have the tree constructed with $k$ points from $A$ and add a $(k+1)$-st point $a\in A\setminus T$. Then there is a unique point $z$ in $T$ closest to $a$; adding the arc $\la z,a\ra$ to $T$ increases the number of vertices and the number of edges by one (if $z$ is a vertex of $T$) or by two  (if $z$ is an interior point of one of the edges of $T$). For any point $b\in T$, $\la b,z\ra\cup \la z,a\ra$ is an arc (and, since $Z_f$ is a tree, the only arc) connecting $b$ and $a$, so adding all the other geodesics connecting $a$ to the first $k$ points used to create $T$ is not necessary.

Thus, we have a finite tree $T\subset Z_f$, with a finite number $E$ of edges. Let us enumerate these edges, $\{\eta_i\}_{i=1}^E$, in such a way that for each $k$ the set $T_k=\bigcup_{i=1}^k \eta_i$ is connected. To do so, we pick an arbitrary edge and label it $\eta_1$, then pick $\eta_2$ among the edges that share a vertex with $\eta_1$, then choose $\eta_3$ among edges sharing a vertex with $\eta_1$ or $\eta_2$ and so on. 
Denote by $\lambda_k$ the length of the edge $\eta_k$. We write $\eta_k=\la u_k, v_k\ra$, where $u_1$ is one of the endpoints of $\eta_1$, and $u_k$, $k\geq 2$, is the unique endpoint of $\eta_k$ that belongs to $T_{k-1}$.

Note that, since the $\eps$-net $A$ is a subset of $T$, every point of $\psi(\overbar{\Omega}_\eps)$ lies in a distance less than $\eps$ to $T$. 
Thus, if $r:Z_f\to T$ is the retraction given in Lemma~\ref{L12}, for any $z\in\psi(\overbar{\Omega}_\eps)$ we have $d_f(r(z),z)=d_f(z,T)<\eps$, so $r$ is $\eps$-close to the identity on $\psi(\overbar{\Omega}_\eps)$.

\noindent
{\bf Embedding $T$ in $\bbbr^E$.}

We embed $T$ in $\bbbr^E$, $w:T\to\bbbr^E$, in the following, inductive way:

We map $u_1$ to the origin in $\bbbr^E$ and embed the edge $\eta_1$ isometrically along the $1$-st coordinate axis:
$$
\eta_1=\la u_1,v_1\ra\xmapsto{~~w~~} [0,\lambda_1 \we_1]
$$
(following the standard conventions, for $a,b\in \bbbr^E$,  we denote by $[a,b]$ the interval with endpoints $a$ and $b$; $\we_1,\ldots, \we_E$ is the standard orthonormal basis in $\bbbr^E$).

The edge $\eta_2$ shares the endpoint $u_2$ with $\eta_1$, 
$u_2\in\{u_1,v_1\}$,
so we already know the value of $w(u_2)$. 
That is, $w(u_2)=0$, if $u_2=u_1$ and $w(u_2)=\lambda_1\we_1$, if $u_2=v_1$.

Then, we map $\eta_2$ to a segment starting at $w(u_2)$ and extending in the $2$-nd coordinate direction:
$$
\eta_2=\la u_2,v_2\ra\xmapsto{~~w~~} [w(u_2),w(u_2)+\lambda_2 \we_2]
\quad
(=[0,\lambda_2\we_2] \text{ or } [\lambda_1\we_1,\lambda_1\we_1+\lambda_2\we_2]),
$$
and so on: once we have $w$ defined on the edges $\eta_1,\ldots,\eta_{k-1}$, we know the value of the embedding $w$ at the endpoint $u_k$ of $\eta_{k}$; then the embedded isometric image of $\eta_{k}$ is the interval in $\bbbr^E$ starting at $w(u_{k})$ and extending in the $k$-th coordinate direction:
$$
\eta_{k}=\la u_k,v_k\ra\xmapsto{~~w~~} [w(u_k),w(u_k)+\lambda_k \we_k].
$$
This way each of the edges $w(\eta_k)$ of $w(T)$ is parallel to $\we_k$ and hence the edges of $w(T)$ are mutually orthogonal. Also, the edges of $w(T)$ form a subset of the edges of the closed $E$-dimensional interval $[0,\lambda_1]\times[0,\lambda_2]\times\cdots\times[0,\lambda_E]$.

More precisely, the edge $w(\eta_i)$ satisfies:
for $1\leq j<i$, there is $\lambda_{ij}\in \{0,\lambda_j\}$ such that the $j$-th coordinate of $w(\eta_i)$ equals $\lambda_{ij}$; the $i$-th coordinate of $w(\eta_i)$ can be any number in $[0,\lambda_i]$; for $i<j\leq E$, the $j$-th coordinate of $w(\eta_i)$ equals $\lambda_{ij}:=0$.

The embedding $w:T\to w(T)\in \bbbr^E$ would be isometric, if we considered $\bbbr^E$ with the $\ell^1$ norm, $|x|_{1}=\sum_{i=1}^E |x_i|$. This norm is less convenient for us than the standard Euclidean norm;  with that norm $w:T\to (\bbbr^E,|\cdot|)$ is 1-Lipschitz, while $w^{-1}:w(T)\to T$ is $\sqrt{E}$-Lipschitz, by the Schwarz inequality.

\noindent
{\bf Projection onto $w(T)$.}

Let $\lambda=\min_{i=1}^E \lambda_i$ denote the minimum of length of edges in $T$. 
Fixing $\eps$ and the $\eps$-net $A\subset Z_f$ determines $\lambda$.
Let $\delta\in (0,\lambda/4]$. The actual value of $\delta$ will depend on $\eps$ and it will be determined later. Let $V_\delta$ be the $\delta$-neighborhood of $w(T)$ in $\bbbr^E$. 

In the next step, we construct a smooth mapping $\rho_\eps:\bbbr^E\to\bbbr^E$, satisfying $\rho_\eps(V_\delta)=w(T)$. It is not difficult to construct a continuous retraction $\hat{\rho}_\eps:V_\delta\to w(T)$, i.e., a continuous map satisfying $\hat{\rho}_\eps(t)=t$ for all $t\in w(T)$. However, in general it is not possible to find a smooth map with that property, so we relax the condition that $\rho_\eps|_{w(T)}=\mathrm{id}$, asking merely that $\rho_\eps(t)$ is sufficiently close to $t$ for all $t\in w(T)$.

For $i=1,2,\ldots, E$, let $\xi_i:\bbbr\to\bbbr$ be a smooth, non-decreasing function satisfying
$$
\xi_i(s)=\begin{cases}
\lambda_i & \text{ for }s>\lambda_i-\delta,\\ 
s & \text{ for }s\in (2\delta,\lambda_i-2\delta),\\
0 & \text{ for }s<\delta,
\end{cases}
$$
so $\xi_i$ maps $(-\delta, \lambda_i+\delta)$ to $[0,\lambda_i]$, $(-\delta,\delta)$ to $0$, $(\lambda_i-\delta,\lambda_i+\delta)$ to $\lambda_i$ and if $s\in (-\delta, \lambda_i+\delta)$, then $|\xi_i(s)-s|<2\delta$.

We define a $C^\infty$-smooth map $\rho_\eps:\bbbr^E\to\bbbr^E$ by
$\rho_\eps(t_1,\ldots,t_E)=(\xi_1(t_1),\ldots,\xi_E(t_E))$.

We will show now that the mapping $\rho_\eps$ maps $V_\delta$ onto $w(T)$. 

Recall that if $s=(s_1,\ldots,s_E)\in w(\eta_i)$, then $s_i\in[0,\lambda_i]$ can be any number, while $s_j=\lambda_{ij}\in \{0,\lambda_j\}$ for $j\neq i$.

If $t\in V_\delta$, then
$t$ is in a $\delta$-neighborhood of one of the edges $w(\eta_i)$.  Hence
$$
\begin{cases}
t_i\in (-\delta,\lambda_i+\delta), & \\
t_j\in (\lambda_{ij}-\delta,\lambda_{ij}+\delta) & \text{for } j\neq i.
\end{cases}
$$
Therefore, $\xi_i(t_i)\in[0,\lambda_i]$ and $\xi_j(t_j)=\lambda_{ij}$ for $j\neq i$, so
$\rho_\eps(t)\in w(\eta_i)$. Since $\xi_i$ maps $(-\delta,\lambda_i+\delta)$ onto $[0,\lambda_i]$, it follows that $\rho_\eps(V_\delta)=w(T)$.

Also,  
\begin{equation}
\label{T31}
|\rho_\eps(t)-t|=\sqrt{\sum_{k=1}^E |\xi_k(t_k)-t_k|^2}\leq 2\sqrt{E}\delta
\quad
\text{for $t\in V_\delta$.}
\end{equation}
Essentially, $\rho_\eps$ acts on $V_\delta$  as the nearest point projection along the edges of $w(T)$, but then collapses a neighborhood of every vertex of $w(T)$ to that vertex.

\noindent
{\bf Approximation and conclusion of the proof.}

Recall that the retraction $r:Z_f\to T$ is $\eps$-close to the identity on $\psi(\overbar{\Omega}_\eps)$: for any $z\in \psi(\overbar{\Omega}_\eps)$ we have $d_f(r(z),z)<\eps$. Thus, by the fact that $\phi$ is $1$-Lipschitz,  
$$
|(\phi\circ r\circ\psi)(x)-f(x)|=|(\phi\circ r\circ\psi)(x)-(\phi\circ\psi)(x)|<\eps
\quad
\text{for all } x\in\overbar{\Omega}_\eps.
$$

 We set $\delta=\min(\eps/(1+\sqrt{E}+2E), \lambda/4)$ and approximate (in a standard way) $g=w\circ r\circ \psi:\Omega\to w(T)\subset \bbbr^E$ with a smooth map $g_\eps:\Omega\to \bbbr^E$, so that $|g_\eps(x)-g(x)|<\delta$ for all $x\in\Omega$. Then $g_\eps(\Omega)$ lies in $V_\delta$, the $\delta$-neighborhood of $w(T)$, which is projected by $\rho_\eps$ back onto $w(T)$ and $|(\rho_\eps\circ g_\eps)(x)-g(x)|<(1+2\sqrt{E})\delta$ by \eqref{T31}.

Since $w^{-1}:w(T)\to T$ is $\sqrt{E}$-Lipschitz and $\phi$ is $1$-Lipschitz,
$\phi\circ w^{-1}:w(T)\to\bbbr^m$ is $\sqrt{E}$-Lipschitz, and we extend it to a Lipschitz map $\tilde{\phi}:\bbbr^E\to\bbbr^m$ (by Lemma~\ref{T19}). Then, we approximate it with a smooth $\phi_\eps:\bbbr^E\to\bbbr^m$, $|\phi_\eps(z)-\tilde{\phi}(z)|<\delta$ for all $z\in \bbbr^E$. Finally, setting $f_\eps=\phi_\eps\circ \rho_\eps\circ g_\eps$, we have for all $x\in \overbar{\Omega}_\eps$
\begin{equation*}
\begin{split}
|f_\eps&(x)-f(x)|=|(\phi_\eps\circ \rho_\eps\circ g_\eps)(x)-f(x)|\\
&\leq|\phi_\eps (\rho_\eps(g_\eps (x)))-\tilde{\phi}(\rho_\eps(g_\eps(x)))|+
|\tilde{\phi}(\rho_\eps(g_\eps(x)))-\tilde{\phi}(g(x))|+|(\tilde{\phi}(g(x))-f(x)|\\
&=|\phi_\eps (\rho_\eps(g_\eps (x)))-\tilde{\phi}(\rho_\eps(g_\eps(x)))|+
|(\phi\circ w^{-1})(\rho_\eps(g_\eps(x)))-(\phi\circ w^{-1})(g(x))|\\
&\qquad \qquad+|(\phi\circ r \circ\psi)(x)-f(x)|\\
& < \delta+\sqrt{E}|(\rho_\eps\circ g_\eps)(x)-g(x)|+\eps\leq  (1+\sqrt{E}+2E)\delta +\eps=2\eps,
\end{split}
\end{equation*}
because $\rho_\eps(g(x))$ and $g(x)$ lie in $w(T)$, where $\tilde{\phi}=\phi\circ w^{-1}$. \\ 
Also, ${(\phi\circ w^{-1})(g(x))=(\phi\circ w^{-1}\circ w\circ r\circ \psi)(x)=(\phi\circ r\circ\psi)(x)}$.

Since for any given compact $K\subset \Omega$ we have $K\subset \Omega_\eps$ for all $\eps\in(0,\eps_K)$, this proves that $f_\eps$ converge uniformly to $f$ on $K$.

It remains to prove that $\rank Df_\eps\leq 1$ everywhere. This, however, follows from the chain rule and the fact that $\rank D\rho_\eps\leq 1$ for all $z\in V_\delta$, because $\rho_\eps(V_\delta)=w(T)$ is 1-dimensional.
\hfill $\Box$

\subsection*{Data Availability}
Data sharing not applicable to this article as no datasets were generated or analysed
during the current study.
\subsection*{Additional Information}
Competing interests: The authors declare no competing financial interests.

\end{document}